\documentclass[11pt,reqno]{amsart}

\usepackage{float}
\usepackage{subcaption}
\usepackage{amsmath,amsthm,amssymb,comment,fullpage}
\usepackage{times}
\usepackage[T1]{fontenc}
\usepackage{mathrsfs}
\usepackage{latexsym}
\usepackage{graphicx}
\usepackage{epsfig}
\usepackage{amsmath,amsfonts,amsthm,amssymb,amscd}
\input amssym.def
\input amssym.tex
\usepackage{color}
\usepackage{hyperref}
\usepackage{url}
\usepackage{breakurl}
\usepackage{comment}
\newcommand{\bburl}[1]{\textcolor{blue}{\url{#1}}}

\newcommand{\prob}[1]{{\rm Prob}\left(#1\right)}

\numberwithin{equation}{section}

\newtheorem{thm}{Theorem}[section]

\newtheorem{lem}[thm]{Lemma}

\newtheorem{defi}[thm]{Definition}

\theoremstyle{plain}

\newtheorem{defn}[thm]{Definition}

\newtheorem{theorem}[thm]{Theorem}

\newcommand\be{\begin{equation}}
\newcommand\ee{\end{equation}}
\newcommand\bea{\begin{eqnarray}}
\newcommand\eea{\end{eqnarray}}
\newcommand\bi{\begin{itemize}}
\newcommand\ei{\end{itemize}}
\newcommand\ben{\begin{enumerate}}
\newcommand\een{\end{enumerate}}
\newcommand\bc{\begin{center}}
\newcommand\ec{\end{center}}
\newcommand\ba{\begin{array}}
\newcommand\ea{\end{array}}

\newcommand{\R}{\ensuremath{\mathbb{R}}}

\newcommand{\hr}[1]{\href{#1}{\url{#1}}}

\renewcommand{\Re}{\operatorname{Re}}

\title{The Inverse Gamma Distribution and Benford's Law}

\author{Rebecca F. Durst}
\email{\textcolor{blue}{\href{mailto:rfd1@williams.edu}{rfd1@williams.edu}}}
\address{Department of
  Mathematics and Statistics, Williams College, Williamstown, MA 01267}

\author{Chi Huynh}
\email{\textcolor{blue}{\href{mailto:huynhngocyenchi@gmail.com}{huynhngocyenchi@gmail.com}}}
\address{School of Mathematics, Georgia Institute of Technology, Atlanta, GA 30332}

\author{Adam Lott}
\email{\textcolor{blue}{\href{mailto:alott@u.rochester.edu}{alott@u.rochester.edu}}}
\address{Department of Mathematics, University of Rochester, Rochester, NY 14627}

\author{Steven J. Miller}
\email{\textcolor{blue}{\href{mailto:sjm1@williams.edu}{sjm1@williams.edu}}}
\address{Department of
  Mathematics and Statistics, Williams College, Williamstown, MA 01267}

\author{Eyvindur A. Palsson}
\email{\textcolor{blue}{\href{mailto:palsson@vt.edu} {palsson@vt.edu}}}
\address{Department of Mathematics, Virginia Tech, Blacksburg, VA 24061}

\author{Wouter Touw}
\email{\textcolor{blue}{\href{w.touw@cmbi.ru.nl}{w.touw@nki.nl}}}
\address{Department of Biochemistry, Netherlands Cancer Institute, Plesmanlaan 121, 1066 CX Amsterdam, the Netherlands}

\author{Gert Vriend}
\email{\textcolor{blue}{\href{Gerrit.Vriend@radboudumc.nl}{Gerrit.Vriend@radboudumc.nl}}}
\address{Radboud University Medical Centre, CMBI, Geert Grooteplein Zuid 26-28,  route 260
6525GA Nijmegen}

\thanks{This work was supported by NSF Grants DMS1265673, DMS1561945, and DMS1347804, Simons Foundation Grant \#360560, Williams College, and the Williams Finnerty Fund. We thank Peter Vijn for suggesting using Benford's law to study protein database submissions, and the referee for helpful comments on an earlier draft.}

\subjclass[2010]{60F05, 11K06 (primary), 60E10, 42A16, 62E15, 62P99 (secondary)}

\keywords{Benford's Law, Inverse Gamma Distribution, digit bias, Poisson Summation}

\date{\today}

\begin{document}

\begin{abstract}
According to Benford's Law, many data sets have a bias towards lower leading digits (about $30\%$ are $1$'s). The applications of Benford's Law vary: from detecting tax, voter and image fraud to determining the possibility of match-fixing in competitive sports. There are many common distributions that exhibit such bias, i.e. they are almost Benford. 
These include the exponential and the Weibull distributions. Motivated by these examples and the fact that the underlying distribution of factors in protein structure follows an inverse gamma distribution, we determine the closeness of this distribution to a Benford distribution as its parameters change.
\end{abstract}

\maketitle

\tableofcontents

\section{Introduction}


\subsection{Motivation}


For a positive integer $B \geq 2$, any positive number $x$ can be written uniquely in base $B$ as $x = S_B(x) \cdot B^{k(x)}$ where  $k(x)$ is an integer and $S_B(x) \in [1,B)$ is called the \textit{significand} of $x$ base $B$.  Benford's Law describes the distribution of significands in many naturally occurring data sets and states that for any $1 \leq s < B$, the proportion of the set with significand at most $s$ is $\log_B(s)$.  In this paper, we examine the behavior of random variables, so we adopt the following definition.



\begin{defn}[Benford's Law]
Let $X$ be a random varialbe taking values in $(0,\infty)$ almost surely.  We say that $X$ follows Benford's Law in base $B$ if, for any $s \in [1,B)$,
\begin{equation}
\prob{S_B(X) \leq s} \ = \ \log_B(s).
\end{equation}
In particular,
\begin{equation}
\prob{\text{first digit of $X$ is $d$}} \ = \ \log_B \left( \frac{d+1}{d} \right).
\end{equation}
\end{defn}

Thus in base 10 about 30\% of numbers have a leading digit of 1, as compared to only about 4.6\% starting with a 9. For an introduction to the theory, as well as a detailed discussion of some of its applications in accounting, biology, economics, engineering, game theory, finance, mathematics, physics, psychology, statistics and voting see \cite{Mi}.

One of  the most  important applications of Benford's  law is  in fraud detection; it has successfully flagged voting irregularities, tax fraud, and embezzlement, to name just a few of its successes. The motivation for this work was to see if a Benford analysis could have detected some fraud on protein structures, as well as serve as a protection against future unscrupulous researchers.

Proteins are the workhorses in all of biology; in plant, human, animal, bacterium, and slime mold, alike. They keep us together, digest our food, make us see, hear, taste, feel, and think, they defend us against pathogens, and they are the target of most existing medicines. Knowledge about the three-dimensional structure of proteins is a prerequisite for research in fields as diverse as drug design, bio-fuel engineering, food processing, or increasing the yield in agriculture.

These three-dimensional structures can be solved with X-ray crystallography, Nuclear Magnetic Resonance, or electron microscopy. Today, most structures are solved with X-ray crystallography. When structures are solved with this technique the experimentalist does not only obtain X, Y and Z coordinates for the atoms, but also a measure of their mobility, which is called the B factor.


After it was detected that 12 of the 14 structures deposited in the PDB protein data bank \cite{BHN} by H. K. M. Murthy were not based on
experimental data (see \bburl{https://www.uab.edu/reporterarchive/71570-uab-statement-on-protein-data-bank-issues}), two of the authors asked the question if their rather anomalous B-factor distributions could have been used to automatically detect the problems (see \bburl{swift.cmbi.ru.nl/gv/Murthy/Murthy_4.html}). In practice B-factor distributions are influenced by experiment conditions and human choices. For example, B factors may fit inverse Gamma distributions translated towards higher values \cite{DNMS, Neg}, or the inverse Gamma fit might be worse when upper and/or lower B factor limits are enforced by the experimentalist. The reported properties of each of the 14 structures were used to find in the PDB a legitimate protein structure of comparable experimental quality, deposition date, size, and  B factor profile. In general, inverse Gamma parameters could be estimated well for both the Murthy structures and the legitimate structures by maximum likelihood estimation when accounting for the translation along the $x$-axis. This suggests the main question of this paper: how close is the inverse Gamma distribution, for various choices of its parameters, to Benford's law? While unfortunately a Benford analysis did  not flag Murthy's structures from legitimate ones, the question of how close this special distribution is to Benford is still of independent interest, and we report on our findings below. This paper is a sequel to \cite{Weibull}, where a similar analysis was done for the three parameter Weibull.

\subsection{Results}

In practice, it is easier to use the following equivalent condition for Benford behavior (see, for example, \cite{Diaconis} or \cite{Mi}), which we reprove here.

\begin{defn}
We say that a random variable $Y$ taking values in $[0,1]$ is equidistributed if, for any $[a,b] \subseteq [0,1]$,
\begin{equation}
\prob{Y \in [a,b]} \ = \ b-a.
\end{equation}
\end{defn}

\begin{theorem}\label{thm:logequi}
A random variable $X$ follows Benford's Law in base $B$ if and only if the random variable $Y := \log_B X \!\mod 1 $ is equidistributed.
\end{theorem}

\begin{proof}We only prove the reverse direction here as that is all we need to prove our main result.  Full details are given in \cite{Diaconis}.  Suppose $Y := \log_B X \!\mod 1$ is equidistributed.  First note that
\begin{align}
Y &\ = \ \log_B(X) \! \mod 1 \nonumber\\
&\ = \ \log_B(S_B(X) \cdot B^{k(X)} ) \!\mod 1 \nonumber\\
&\ = \ \log_B(S_B(X)) + \log_B(B^{k(X)}) \!\mod 1 \nonumber\\
&\ = \ \log_B(S_B(X)).
\end{align}
Then, taking $a=0$, $b=\log_B(p)$ in the definition of equidistribution, we get
\begin{equation}
\prob{\log_B(S_B(X)) \in [0, \log_B(p)]} \ = \ \log_B(p).
\end{equation}
Exponentiating gives
\begin{equation}
\prob{S_B(X) \in [1,p]} \ = \ \log_B(p),
\end{equation}
which is exactly the statement of Benford's Law.
\end{proof}


In this paper, we examine the behavior of a random variable drawn from the inverse gamma distribution.  For fixed parameters $\alpha$, $\beta > 0$, this distribution has density defined by
\begin{equation}\label{eq:invgamPDF}
f(x;\alpha,\beta) \ = \ \frac{\beta^{\alpha}}{\Gamma(\alpha)} x^{-\alpha - 1} \exp\left(\frac{-\beta}{x}\right)
\end{equation}
and cumulative distribution function
\begin{equation} \label{eq: invgamCDF}
F(x;\alpha,\beta) \ = \ \frac{1}{\Gamma(\alpha)} \int_{\beta/x}^{\infty} t^{\alpha-1}e^{-t}
\, dt\end{equation}
Let $X_{\alpha, \beta}$ be a random variable distributed according to (\ref{eq:invgamPDF}) and let $F_B$ be the cumulative distribution function of $\log_B(X_{\alpha, \beta})\! \mod 1$.  By Theorem \ref{thm:logequi}, the assertion that $X_{\alpha, \beta}$ follows Benford's Law is equivalent to saying that $F_B(z) = z$ for all $z \in [0,1]$.  In this paper, we investigate when the deviations of $F_B(z)$ from $z$ are small, i.e., when $X_{\alpha, \beta}$ approximately follows Benford's Law.  We do this by deriving a series expansion for $F_B'(z)$ of the form $1 + (\text{error term})$, where the error term can be computed to great accuracy, and then integrating in order to return to the cumulative distribution function, $F_B(z)$.


In Section \ref{sec: SeriesRep}, we derive our series representation for $F_B'(z)$.  In Section \ref{sec: TailBound}, we give bounds for the tail of the series, showing that the series can be computed to great accuracy by computing only the first few terms. This result is built upon in Appendix \ref{app: bound1}.  In Section \ref{sec: Plots}, we use this result to generate some plots illustrating the Benfordness of the inverse gamma distribution as a function of $\alpha$ and $\beta$.

\section{Series representation for $F_B'(z)$}
\label{sec: SeriesRep} 
Before beginning the analysis, we first note a useful invariant property of the Benfordness of this distribution.
\begin{lem} \label{prop: BetaInvariance}
For any $\alpha, \beta > 0$ and $z \in [0,1]$,
\begin{equation} \label{eq: BetaInvariance}
\prob{\log_B S_B(X_{\alpha,\beta}) \leq z} \ = \ \prob{\log_B S_B(X_{\alpha,B \cdot \beta}) \leq z}.
\end{equation}
In other words, the deviation from Benford's law of the inverse Gamma distribution doesn't change if we scale $\beta$ by a 
factor of $B$.
\end{lem}

\begin{proof}
Scaling $\beta$ by a 
factor of B yields
\begin{align}
 \nonumber \prob{\log_B S_B(X_{\alpha,B \cdot \beta}) \leq z} &\ =  \ \sum_{k = -\infty}^{\infty} \prob{\log_{B}X_{\alpha,B \cdot \beta} \in [k,z+k]}  \\
 &\ = \ \sum_{k = -\infty}^{\infty} \prob{X_{\alpha,B \cdot \beta}\in [B^{k},B^{z+k}]},
\end{align}

which, by (\ref{eq: invgamCDF}), is
\begin{align}\label{lem: BenfordDev}
 \nonumber &\ = \ \frac{1}{\Gamma(\alpha)} \sum_{k = -\infty}^{\infty} \left( \int_{ B \cdot \beta / B^{z+k}}^{\infty} t^{\alpha - 1} e^{-t} dt - \int_{ B \cdot \beta / B^{k}}^{\infty }t^{\alpha - 1} e^{-t} dt \right) \\
\nonumber &\ = \ \frac{1}{\Gamma(\alpha)} \sum_{k = -\infty}^{\infty} \int_{ B \cdot \beta / B^{z+k}}^{ B \cdot \beta / B^{k}} t^{\alpha - 1} e^{-t} dt \\
 \nonumber &\ = \ \frac{1}{\Gamma(\alpha)} \sum_{k = -\infty}^{\infty} \int_{ \beta / B^{z+k-1}}^{ \beta / B^{k-1}} t^{\alpha - 1} e^{-t} dt \\
 &\ = \ \prob{S_B(X_{\alpha,\beta}) \leq z}.
\end{align}

Thus, scaling $\beta$ by a power of $B$ only results in shifting $k$. Since we take an infinite sum over $k$, this shift does not change the final value of the probability. As a consequence of this, it is clear that scaling $\beta$ by any power of $B$ will yield the same result, shifting $k$ by that power.
\end{proof}
\noindent Thus it suffices to study $1 \leq \beta < B$.


To show that the deviations of $F_B(z)$ from $z$ are small, it is easier in practice to show that $F_B'(z)$ is close to 1, and then integrate. We derive a series representation for $F_B'(z)$, but first, we state a useful property of Fourier transforms (see, for example, \cite{Stein}).

Throughout the course of this paper, we define the Fourier transform as follows.
\begin{defi}[Fourier Transform]
Let $f \in L^1(\R)$.  Define the Fourier transform $\hat{f}$ of $f$ by
\begin{equation}\label{eq: fourier}
\hat{f}(\xi) \ := \ \int_{-\infty}^{\infty}f(x)e^{-2\pi i x \xi}dx.
\end{equation}
\end{defi}
Furthermore, we will occasionally use the notation
\begin{equation}\label{eq: altfourier}
\mathcal{F}(f(x))(\xi) \ := \ \hat{f}(\xi).
\end{equation}


Our main tool is the Poisson summation formula, which we state here in a weak form (see Theorem 3.1 of \cite{Weibull} for a more detailed explanation).
\begin{theorem}[Poisson Summation]
Let $f$ be a function such that $f$, $f'$, and $f''$ are all $O(x^{-(1+\eta)})$  as $x \to \infty$ for some $\eta > 0$.  Then
\begin{equation} \label{poisson-sum}
\sum_{k=-\infty}^{\infty} f(k) \ = \ \sum_{k=-\infty}^{\infty} \hat{f}(k).
\end{equation}
\end{theorem}


\begin{theorem}\label{thm: series}
Let $\alpha$, $\beta > 0$ be fixed and let $B \geq 3$ be an integer.  Let $X_{\alpha,\beta}$ be a random variable distributed according to equation (\ref{eq:invgamPDF}). For $z\in [0,1]$, let $F_{B}(z)$ be the cumulative distribution function of $\log_{B}(X_{\alpha,\beta}) \!\mod 1$.  Then $F_{B}'(z)$ is given by
\begin{equation} \label{eq: series}
F_{B}'(z) \ = \ 1 + \frac{2}{\Gamma(\alpha)} \sum_{k=1}^{\infty} \Re \left( e^{2\pi ik(\log_{B}{\beta-z})} \Gamma \left( \alpha-\frac{2\pi ik}{\log{B}} \right) \right).
\end{equation}

\end{theorem}
\begin{proof}
By the argument leading to (\ref{lem: BenfordDev}),
\begin{equation}
F_B(z) \ = \ \frac{1}{\Gamma(\alpha)} \sum_{k = -\infty}^{\infty} \int_{\frac{\beta}{B^{z+k}}}^{\frac{\beta}{B^{k}}} t^{\alpha - 1} e^{-t} dt.
\end{equation}
We want to show that this series converges uniformly for $z \in [0,1]$.  Let
\begin{equation}
g_0(z) \ := \ \int_{\frac{\beta}{B^{z}}}^{\beta} t^{\alpha - 1} e^{-t} dt
\end{equation}
and for $k \geq 1$,
\begin{equation}
g_k(z) \ := \ \int_{\frac{\beta}{B^{z+k}}}^{\frac{\beta}{B^{k}}} t^{\alpha - 1} e^{-t} dt + \int_{\frac{\beta}{B^{z-k}}}^{\frac{\beta}{B^{-k}}} t^{\alpha - 1} e^{-t} dt.
\end{equation}
Notice that each $g_k$ is monotonically increasing in $z$ and positive for $z \in [0,1]$.  So we have $g_k(z) \leq g_k(1)$ for $z \in [0,1]$ and
\begin{equation}
\sum_{k = 0}^{\infty} g_k(1) \ = \ F_B(1) \ = \ 1.
\end{equation}
Thus the Weierstrass $M$-test implies that
\begin{equation}
\sum_{k = 0}^{\infty} g_k(z) \ = \ \sum_{k = -\infty}^{\infty} \int_{\frac{\beta}{B^{z+k}}}^{\frac{\beta}{B^{k}}} t^{\alpha - 1} e^{-t} dt
\end{equation}
converges uniformly for $z \in [0,1]$.

Since the convergence is uniform, we can differentiate term by term to obtain
\begin{equation} \label{eq: derivative}
F_{B}'(z) \ = \ \frac{1}{\Gamma(\alpha)} \sum_{k=-\infty}^{\infty} \left( \frac{\beta}{B^{z+k}} \right)^{\alpha} \exp \left(\frac{-\beta}{B^{z+k}} \right) \log{B}.
\end{equation}
Applying Poisson summation to (\ref{eq: derivative}) gives
\begin{equation}
F_{B}'(z) \ = \ \frac{1}{\Gamma(\alpha)} \sum_{k=-\infty}^{\infty} \int_{-\infty}^{\infty} \left( \frac{\beta}{B^{z+t}} \right)^{\alpha} \exp \left(\frac{-\beta}{B^{z+t}} \right) \log{B} \ \exp (-2\pi itk) \ dt.
\end{equation}
We now let $x=\frac{\beta}{B^{z+t}}$ and $dx=\frac{-\beta}{B^{z+t}}\log{B}\ dt$ so that we have
\begin{align}
F_{B}'(z)\nonumber &\ = \frac{1}{\Gamma(\alpha)} \ \sum_{k=-\infty}^{\infty} \int_{0}^{\infty} x^{\alpha-1} \exp \left( -2\pi ik \left( \frac{\log{\frac{\beta}{B^{z}x}}}{\log{B}} \right) \right) e^{-x} dx \\
\nonumber &\ = \ \frac{1}{\Gamma(\alpha)} \sum_{k=-\infty}^{\infty} \int_{0}^{\infty} x^{\alpha-1}\left( \frac{\beta}{B^{z}x} \right)^{\frac{-2\pi ik}{\log{B}}} e^{-x} dx \\
\nonumber &\ = \ \frac{1}{\Gamma(\alpha)} \sum_{k=-\infty}^{\infty} \left( \frac{\beta}{B^{z}} \right)^{\frac{-2\pi ik}{\log{B}}} \int_{0}^{\infty} x^{\alpha-1+\frac{2\pi ik}{\log{B}}} e^{-x} dx \\
 &\ = \ \frac{1}{\Gamma(\alpha)} \sum_{k=-\infty}^{\infty} \left( \frac{\beta}{B^{z}} \right)^{\frac{-2\pi ik}{\log{B}}} \Gamma \left( \alpha+\frac{2\pi ik}{\log{B}} \right).
\end{align}

Note that $\left( \frac{\beta}{B^{z}} \right)^{2\pi i\theta} = \exp \left( 2\pi i \theta \log{\frac{\beta}{B^{z}}} \right)$, so our sum becomes

\begin{equation} \label{eq: HalfSimpDeriv}
F_B'(z) \ = \ \frac{1}{\Gamma(\alpha)} \sum_{k=-\infty}^{\infty} \exp \left( \frac{-2\pi ik\log{\frac{\beta}{B^{z}}}}{\log{B}} \right) \Gamma \left( \alpha+\frac{2\pi ik}{\log{B}} \right).
\end{equation}
This form of our sum will become useful in a later proof, but for the purposes of this theorem, we further simplify our derivative and point out that the $k=0$ term in \eqref{eq: HalfSimpDeriv} is equal to 1. Thus our equation becomes
\begin{align}
\nonumber F_{B}'(z) \ = \ 1 + \frac{1}{\Gamma(\alpha)} &\Bigg[ \sum_{k=1}^{\infty} \exp \left( \frac{2\pi ik\log{\frac{\beta}{B^{z}}}}{\log{B}} \right) \Gamma \left( \alpha-\frac{2\pi ik}{\log{B}} \right) \\
\nonumber &+ \sum_{k=1}^{\infty} \exp \left( \frac{-2\pi ik\log{\frac{\beta}{B^{z}}}}{\log{B}} \right) \Gamma \left( \alpha+\frac{2\pi ik}{\log{B}} \right) \Bigg] \\
\nonumber \ = \ 1 + \frac{1}{\Gamma(\alpha)} & \Bigg[ \sum_{k=1}^{\infty} \exp \left( 2\pi ik(\log_{B}{\beta-z}) \right) \Gamma \left( \alpha-\frac{2\pi ik}{\log{B}} \right) \\
 &+ \exp \left( -2\pi ik(\log_{B}{\beta-z}) \right) \Gamma \left( \alpha+\frac{2\pi ik}{\log{B}} \right) \Bigg].
\end{align}


Finally, using the identity that $\overline{\Gamma(a+ib)}=\Gamma(a-ib)$ for real numbers $a$ and $b$, we have
\begin{equation} \label{eq: SimpDeriv}
F_{B}'(z) \ = \ 1 + \frac{2}{\Gamma(\alpha)} \sum_{k=1}^{\infty}\Re \left( e^{2\pi ik(\log_{B}{\beta-z})} \Gamma \left( \alpha-\frac{2\pi ik}{\log{B}} \right) \right).
\end{equation}

\end{proof}

\section{Bounding the truncation error}
\label{sec: TailBound}


A key tool for the analysis in \cite{Weibull} is the identity
\begin{equation}
| \Gamma(1+ix) |^2 \ = \ \frac{\pi x}{\sinh(\pi x)}
\label{eq: GammaIdentity}
\end{equation}
for real $x$.  Examining \eqref{eq: SimpDeriv}, it is clear that when $\alpha = 1$, our analysis of the truncation error is similar to that of \cite{Weibull}.  Since the bound resulting from such analysis in the case of $\alpha = 1$ is tighter than the bound for an arbitrary $\alpha$, we have included the proof in the appendix.  However, when $\alpha \neq 1$, the identity \eqref{eq: GammaIdentity} is no longer applicable, so a new approach is needed to bound the tails of the series expansion.  We have the following bound on the truncation error.

\begin{theorem} \label{thm: TruncationError}
Let $F_B'(z)$ be as in (\ref{eq: HalfSimpDeriv}).  Let $E_M(z)$ denote the two-sided tail of the series expansion, i.e.,
\begin{equation}
E_M(z) \ := \ \frac{1}{\Gamma(\alpha)}\sum_{|k| \geq M} \exp \left( \frac{-2\pi ik\log{\frac{\beta}{B^{z}}}}{\log{B}} \right) \Gamma \left( \alpha+\frac{2\pi ik}{\log{B}} \right).
\end{equation}

\begin{enumerate}
\item \label{part: ZDependentBound}
We have
\begin{equation} \label{eq: ZDependentBound}
|E_M(z)| \ \leq \ \frac{B^{\alpha(1-z)}\beta}{\Gamma(\alpha)} \left( \int_{B^M}^{\infty} e^{-x}x^{\alpha -1}dx + \frac{1}{\alpha} B^{-M\alpha}  \right).
\end{equation}
\vspace{1mm}
\item \label{part: ZIndependentBound}
This is bounded uniformly on $z \in [0,1]$ by the constant
\begin{equation} \label{eq: ZIndependentBound}
|E_M(z)| \ \leq \ \frac{B^{\alpha}\beta}{\Gamma(\alpha)} \left( \int_{B^M}^{\infty} e^{-x}x^{\alpha -1}dx + \frac{1}{\alpha} B^{-M\alpha} \right).
\end{equation}
\vspace{1mm}
\vspace{1mm}
\item \label{part: EpsilonThreshold}
Furthermore, for any $\epsilon > 0$, in order to have $|E_M(z)| < \epsilon$ in (\ref{eq: ZIndependentBound}) it suffices to take
\begin{equation} \label{eq: EpsilonThreshold}
M \ > \ \max \left( \alpha+1, \ -\log_B \left(\frac{\epsilon \cdot \Gamma(\alpha)}{2B^{\alpha} \beta} \right) \right).
\end{equation}
\end{enumerate}
\end{theorem}

\begin{proof}[Proof of part (\ref{part: ZDependentBound}): locally bounding the truncation error]
\par

We begin with (\ref{eq: HalfSimpDeriv}).

Let $\phi(z)=\log{\frac{\beta}{B^{z}}}$.
We have
\begin{equation}
E(z) \ := \ F_B'(z) - 1 \ = \ \frac{1}{\Gamma(\alpha)} \sum_{|k| \geq 1} \exp \left( -2\pi \frac{ik\phi(z)}{\log{B}} \right) \Gamma \left( \alpha + 2\pi \frac{ik}{\log{B}} \right).
\end{equation}
Furthermore, given $\Gamma(a+2\pi i b) =\int_{0}^{\infty}e^{-x}x^{a+2\pi i b-1}dx$, we may perform a change of variables and let 
 $x=e^{-u}$ so that we get
\begin{equation}
\Gamma(a+2\pi b i) \ = \ \int_{-\infty}^{\infty}e^{-e^{-u}}e^{- au}e^{-2\pi ibu}du \ = \ \mathcal{F} \left(   e^{-e^{- u}}e^{- au} \right)(b),
\end{equation}
where $\mathcal{F}(\cdot)$ denotes the Fourier transform, as stated in \eqref{eq: altfourier}.  This transforms our sum into the sum of terms of the form 
\begin{align}
\nonumber &\exp \left( \frac{-2\pi ik\phi(z)}{\log{B}} \right) \Gamma \left( \alpha+2\pi \frac{ik}{\log{B}} \right) \\
 &\ = \ \exp \left( \frac{-2\pi ik\phi(z)}{\log{B}} \right)\left[ \mathcal{F}  \left(  e^{-e^{- u}}e^{-\alpha u} \right) \left( \frac{k}{\log{B}} \right)\right].
\end{align}
Suppose $s\in L^{1}(\mathbb{R})$, $P>0$, and $t\in \mathbb{R}$. Define
\begin{equation}
g(x) \ \equiv \ s(Px+t).
\end{equation}
The scaling and frequency shift properties of Fourier transforms then yield
\begin{equation}
\hat{g}(\xi) \ = \ \frac{1}{P}\exp\bigg(\frac{2\pi i k t}{P}\bigg)\hat{s} \left( \frac{\xi}{P} \right).
\end{equation}
Thus, if $g$ meets the conditions required for Poisson summation, we have

\begin{equation}
P\sum_{n \in \mathbb{Z}}s(t +nP) \ = \ \sum_{k\in \mathbb{Z}}\exp\bigg( \frac{2\pi i k t}{P}\bigg)\mathcal{F}(s) \left( \frac{k}{P} \right).
\end{equation}
Therefore, letting $s=e^{-e^{-u}}e^{-\alpha u}$, $P=\log{B}$, and $t=-\phi(z)$, we have
\begin{align}
\nonumber E(z) \ &= \frac{1}{\Gamma(\alpha)} \sum_{|k| \geq 1}\mathcal{F}(s) \left( \frac{k}{P} \right) e^{2\pi i \frac{k}{P}t}\frac{1}{P} \\
\nonumber &= \ \left(\frac{1}{\Gamma(\alpha)} \sum_{k \in \mathbb{Z}}\mathcal{F}(s) \left( \frac{k}{P} \right) e^{2\pi i \frac{k}{P}t}\frac{1}{P}\right) \ - \ \frac{1}{\Gamma(\alpha)}  \\
\nonumber &\leq \frac{P}{\Gamma(\alpha)} \sum_{k\in \mathbb{Z}} s(t+kP) \\
\nonumber &= \ \frac{P}{\Gamma(\alpha)} \sum_{k \in \mathbb{Z}} \exp \left( -e^{\phi(z)} e^{-k \log B} \right) e^{\alpha \phi(z)} e^{-\alpha k \log B} \\
\ &= \ \frac{P}{\Gamma(\alpha)} \sum_{k \in \mathbb{Z}} \exp \left( - \beta B^{-z} B^{-k} \right) e^{\alpha \phi(z)} e^{-\alpha k \log B}.
%
%
%
\end{align}
Recall that we are only working in the range $1 \leq \beta < B$, $0 \leq z  \leq 1$.  Thus for each $z$ we have $\beta B^{-z}B^{-k} \geq \beta B^{-1} B^{-k} = \beta B^{-k-1}$.  Thus the equation above reduces to
\begin{align}
E(z) \ &\leq \ \frac{P}{\Gamma(\alpha)} \sum_{k \in \mathbb{Z}} \exp \left( - \beta B^{-z} B^{-k} \right) e^{\alpha \phi(z)} e^{-\alpha k \log B} \\
&\leq \ \frac{P}{\Gamma(\alpha)} \sum_{k \in \mathbb{Z}} \exp \left( - \beta B^{-k-1} \right) e^{\alpha \phi(z)} e^{-\alpha k \log B} \\
&= \ \frac{(\log B) e^{\alpha \phi(z)}}{\Gamma(\alpha)} \sum_{k \in \mathbb{Z}} \exp \left( -\beta e^{-(k+1)\log B} \right) e^{-\alpha k \log B}.
\end{align}

We now concentrate on the truncation error $E_M(z)$, given by
\begin{equation}
E_M(z) \ \leq \ \frac{(\log{B}) e^{\alpha \phi(z)} }{\Gamma(\alpha)} \sum_{|k| \leq M} \exp \left( -\beta e^{-(k+1) \log B} \right) e^{-\alpha k\log{B}}.
\end{equation}
We bound our sums by integrals and perform a change of variables, letting $x=e^{-(k+1)\log{B}}$ and $dx = -(\log B) e^{-(k+1)\log{B}} dk$.  This yields
\begin{align}
\nonumber |E_M(z)| \ &\leq \ \frac{e^{\alpha \phi(z)}}{\Gamma(\alpha)} \left( \int_{B^M}^{\infty} e^{-x}x^{\alpha -1} e^{\alpha \log B}dx +\int_{0}^{B^{-M}} e^{-x}x^{\alpha -1} e^{\alpha \log B} dx \right) \\
\nonumber &\leq \ \frac{ e^{\alpha \log B}e^{\alpha \phi(z)}}{\Gamma(\alpha)} \left( \int_{B^M}^{\infty} e^{-x}x^{\alpha -1}dx +\int_{0}^{B^{-M}} x^{\alpha -1}dx \right) \\
 &\leq \ \frac{B^{\alpha(1-z)}\beta}{\Gamma(\alpha)} \left( \int_{B^M}^{\infty} e^{-x}x^{\alpha -1}dx + \frac{1}{\alpha} B^{-M\alpha}  \right),
\end{align}
which is (\ref{eq: ZDependentBound}), thus proving (\ref{part: ZDependentBound}).
\\
\\

\textit{Proof of part (\ref{part: ZIndependentBound}): uniformly bounding the truncation error for $z \in [0,1]$.}
To get (\ref{eq: ZIndependentBound}), we simply maximize (\ref{eq: ZDependentBound}) with respect to $z$.  Set
\begin{equation}
g(z) \ = \ B^{\alpha(1-z)},
\end{equation}
and note that
\begin{equation} \label{eq: EDerivative}
g'(z) \ = \ B^{\alpha} B^{-\alpha z} (-\alpha) \log B,
\end{equation}
which is negative for $z \in [0,1]$.  Hence $g$ is decreasing on $z \in [0,1]$, so $g$ is maximized at $z=0$, yielding
\begin{equation}
|E_M(z)| \ \leq \ \frac{B^{\alpha}\beta}{\Gamma(\alpha)} \left( \int_{B^M}^{\infty} e^{-x}x^{\alpha -1}dx + \frac{1}{\alpha} B^{-M\alpha} \right).
\end{equation}
This proves (\ref{part: ZIndependentBound}). \\

\textit{Proof of part (\ref{part: EpsilonThreshold})}.  Fix an $\epsilon > 0$ and suppose
\begin{equation} \label{eq: MBound}
M \ > \ \max \left( \alpha+1, \ -\log_B \left(\frac{\epsilon \cdot \Gamma(\alpha)}{2B^{\alpha} \beta} \right) \right).
\end{equation} In particular, because $B \geq 3$ this implies that $B^M > e^{\alpha+1}$.  Since $x/\log x$ is an increasing function and $B^M/\log(B^M) > e^{\alpha+1}/(\alpha+1) > \alpha+1$, this shows that $x / \log x > \alpha + 1$ for all $x \geq B^M$, which implies that
\begin{equation} \label{eq: ConditionOne}
e^{-x} x^{\alpha - 1} \ \leq \ 1/x^2.
\end{equation}
Equation (\ref{eq: MBound}) also implies that
\begin{equation} \label{eq: ConditionTwo}
\frac{1}{\alpha} B^{-M \alpha} + B^{-M} \ < \ 2B^{-M} \ < \ \frac{\epsilon \cdot \Gamma(\alpha)}{B^{\alpha} \beta}.
\end{equation}
Combining (\ref{eq: ConditionOne}) and (\ref{eq: ConditionTwo}) with (\ref{eq: ZIndependentBound}), we have the bound
\begin{align}
|E_M(z)| \ &< \ \frac{B^{\alpha} \beta}{\Gamma(\alpha)} \left( \frac{1}{\alpha} B^{-M \alpha} + \int_{B^M}^{\infty} \frac{1}{x^2} dx \right) \nonumber \\
&< \ \frac{B^{\alpha} \beta}{\Gamma(\alpha)} \left( \frac{1}{\alpha} B^{-M \alpha} + B^{-M} \right) \nonumber \\
&< \ \frac{B^{\alpha} \beta}{\Gamma(\alpha)} \cdot \frac{\epsilon \cdot \Gamma(\alpha)}{B^{\alpha} \beta} \ = \ \epsilon.
\end{align}
\end{proof}

\section{Plots and analysis}
\label{sec: Plots}


Using Theorem \ref{thm: TruncationError} allows us to easily compare $F_B(z)$, the CDF of $\log X_{\alpha, \beta}$, with $z$, the Benford CDF.  We simply integrate (\ref{eq: SimpDeriv}) from $0$ to $z$, yielding
\begin{equation} \label{eq: newCDF}
F_B(z) \ = \ z + \frac{1}{\Gamma(\alpha)} \sum_{|k| \geq 1} \Gamma \left( \alpha + \frac{2 \pi i k}{\log B} \right) \frac{1}{2 \pi i k} e^{ -2 \pi i k \log_B(\beta)} \left( e^{ 2 \pi i k z} - 1 \right).
\end{equation}
We now use Theorem \ref{thm: TruncationError} in the following way.  Fix an $\epsilon > 0$.  Then part \eqref{part: EpsilonThreshold} of Theorem \ref{thm: TruncationError} allows us to quickly compute the value of $|F_B'(z) - 1|$ to within $\epsilon$ of the true value.  Thus, after integrating, since we are only working on $z \in [0,1]$, the mean value theorem guarantees that we now know $|F_B(z) - z|$ to within $\epsilon$ of the true value.  In short, Theorem \ref{thm: TruncationError} allows us to obtain very good estimates for $|F_B(z) - z|$ by taking only the first few terms of the sum in \eqref{eq: newCDF}, which makes calculating the deviation more computationally feasible.  To measure the closeness to Benford of the distribution, we use the quantity
\begin{equation}
\max_{z \in [0,1]} |F_B(z) - z|.
\end{equation}
In Figure \ref{fig: ContourPlotsMaxDeviation}, we illustrate this quantity as a function of $\alpha$ and $\beta$ with $B=10$ fixed.  In Figures \ref{fig: FBexamples} and \ref{fig: FBexamples2} we show examples of the graph of $F_B(z)$ for different values of $\alpha$ and $\beta$.  The emergent trend is that as $\alpha$ increases, the distribution gets farther away from Benford, and the Benfordness is largely independent of $\beta$.  This behavior is similar to that of the Weibull distribution exhibited in \cite{Weibull}.

\begin{figure}[h]
\begin{subfigure}{.4\textwidth}
\includegraphics[scale=.7]{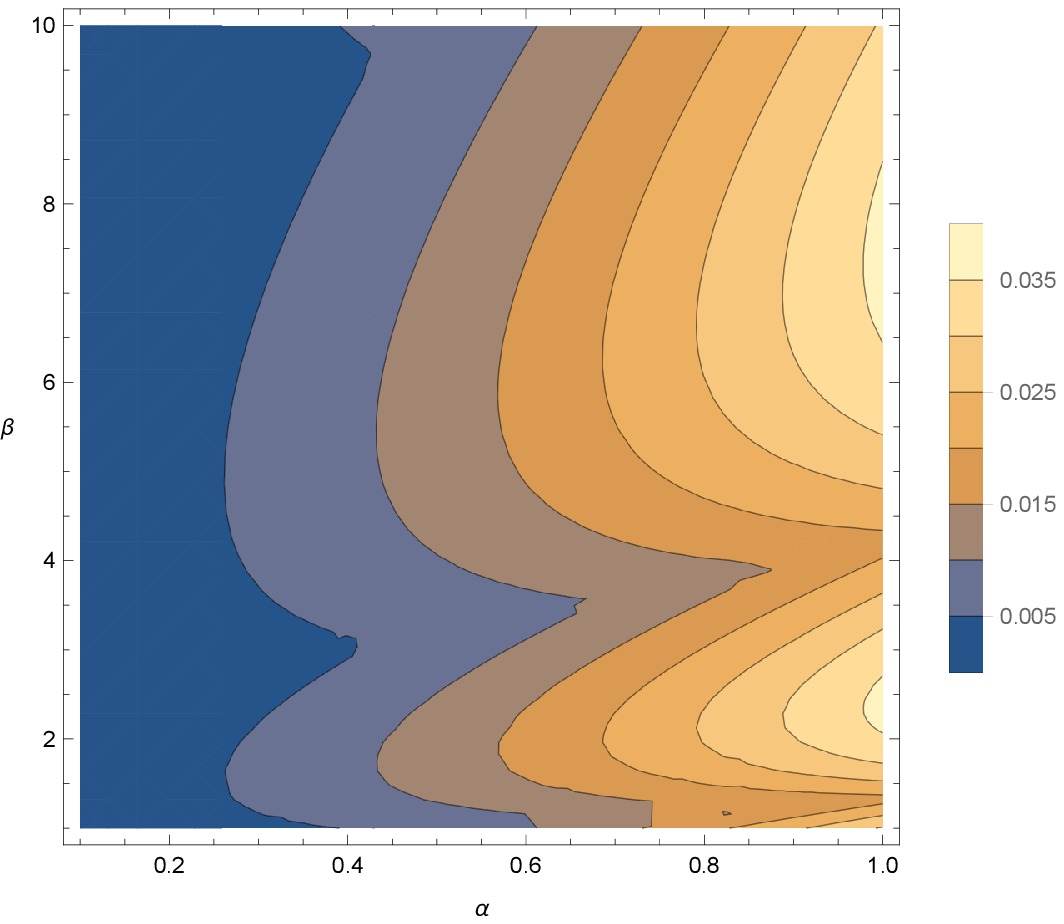}
\caption{$\alpha \in [0.1,1]$}
\end{subfigure}
\hspace{15mm}
\begin{subfigure}{.4\textwidth}
\includegraphics[scale=.7]{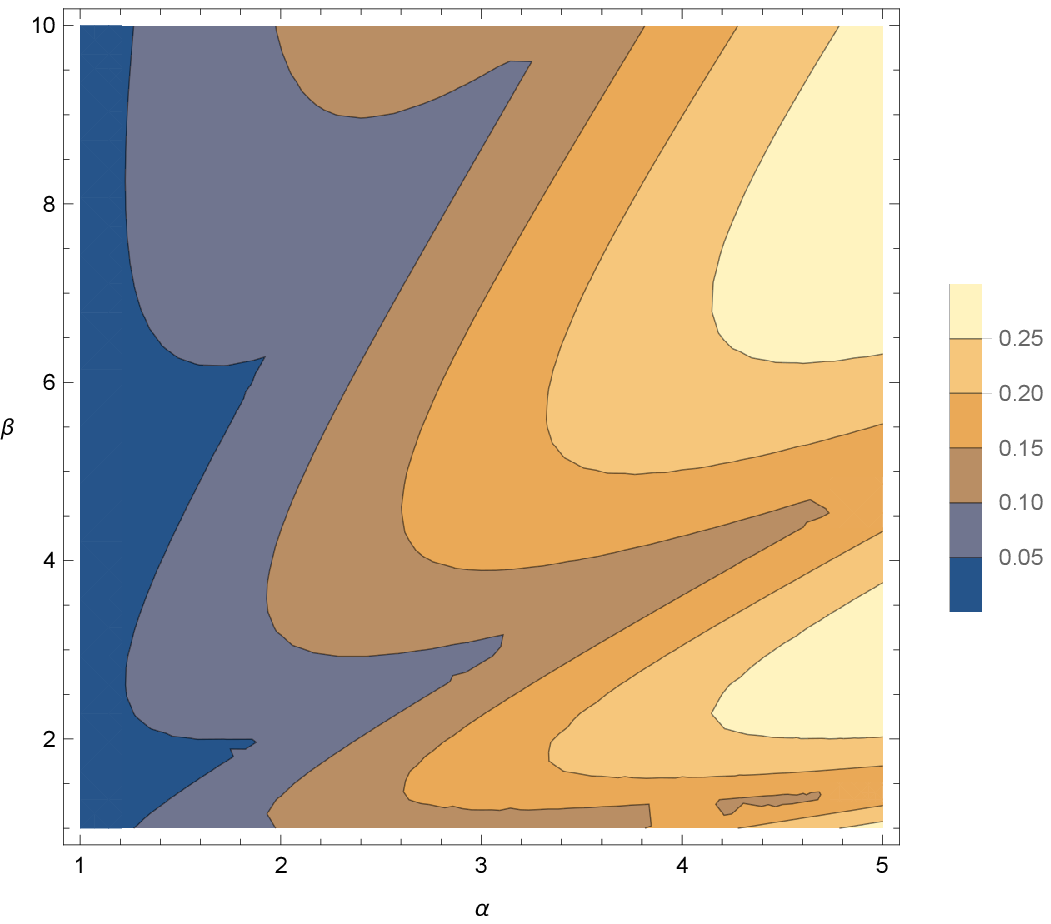}
\caption{$\alpha \in [1,5]$}
\end{subfigure}
\\
\vspace{5mm}
\begin{subfigure}{.4\textwidth}
\bc
\includegraphics[scale=.7]{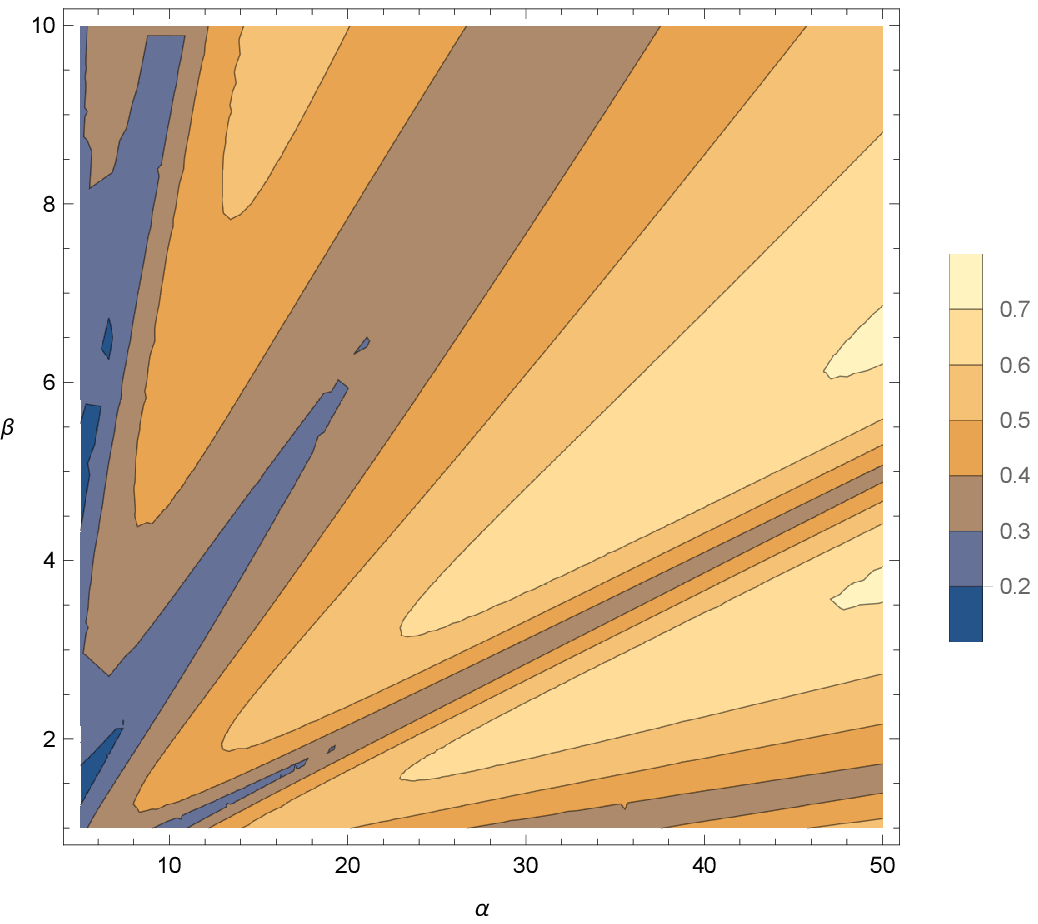}
\caption{$\alpha \in [5,50]$}
\ec
\end{subfigure}
\caption{Contour plots of the quantity $\max_{z \in [0,1]} |F_B(z) - z|$ (see \eqref{eq: newCDF}) as a function of $\alpha$ and $\beta$ with $B=10$ fixed.  Using part \eqref{part: EpsilonThreshold} of Theorem \ref{thm: TruncationError}, we have made the displayed values accurate to within $\epsilon = 0.001$.  Notice that the error is large for large $\alpha$, meaning that the inverse gamma distribution only approximates Benford behavior for small $\alpha$.  Also notice that $\beta$ has less of an effect on the error.}
\label{fig: ContourPlotsMaxDeviation}
\end{figure}

\begin{figure}[h]
\begin{subfigure}{.4\textwidth}
\includegraphics[scale=.7]{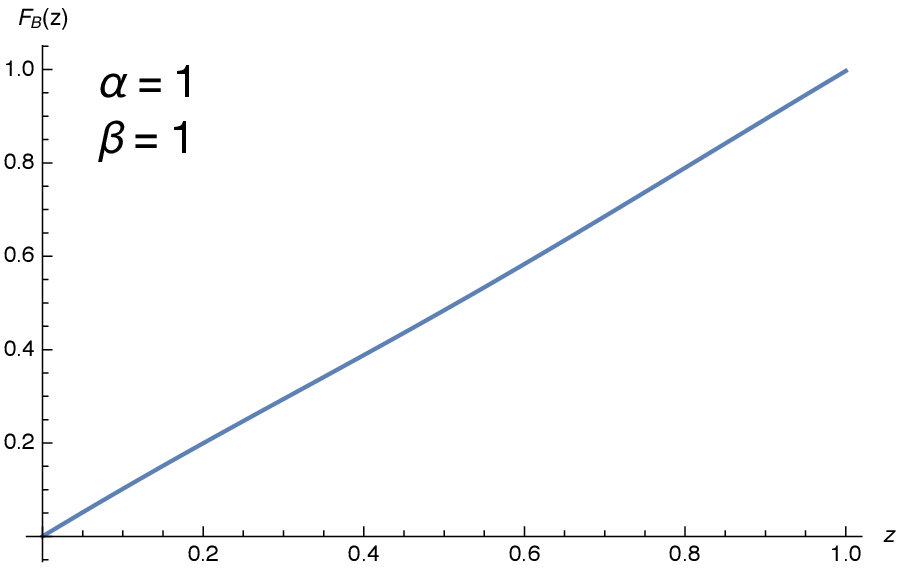}
\caption{$\alpha=1$, $\beta = 1$}
\end{subfigure}
\hspace{15mm}
\begin{subfigure}{.4\textwidth}
\includegraphics[scale=.7]{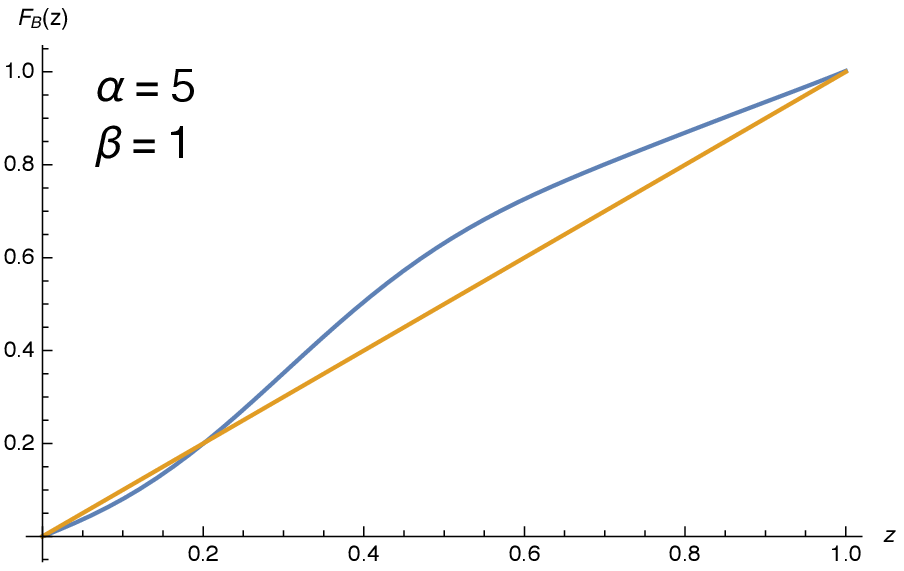}
\caption{$\alpha=5$, $\beta = 1$}
\end{subfigure}
\\
\begin{subfigure}{.4\textwidth}
\includegraphics[scale=.7]{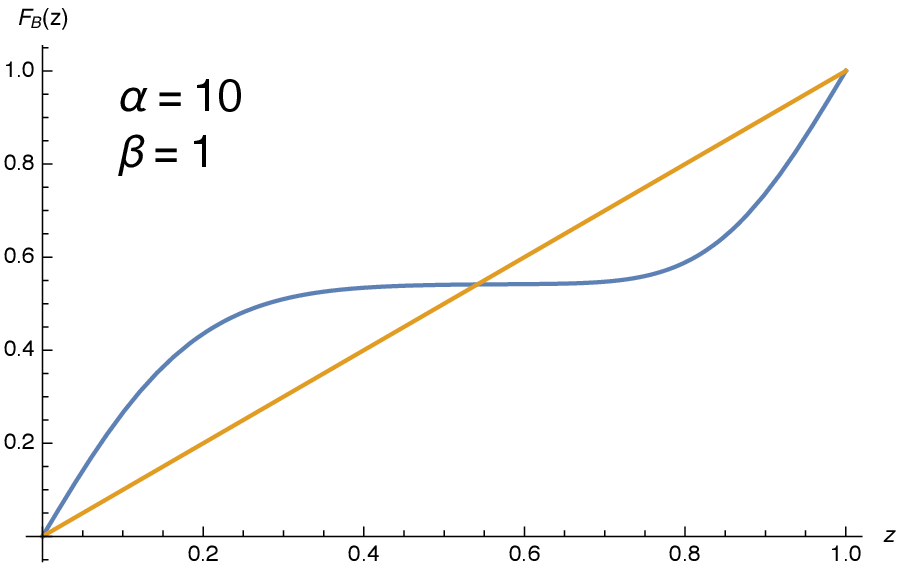}
\caption{$\alpha=10$, $\beta = 1$}
\end{subfigure}
\hspace{15mm}
\begin{subfigure}{.4\textwidth}
\includegraphics[scale=.7]{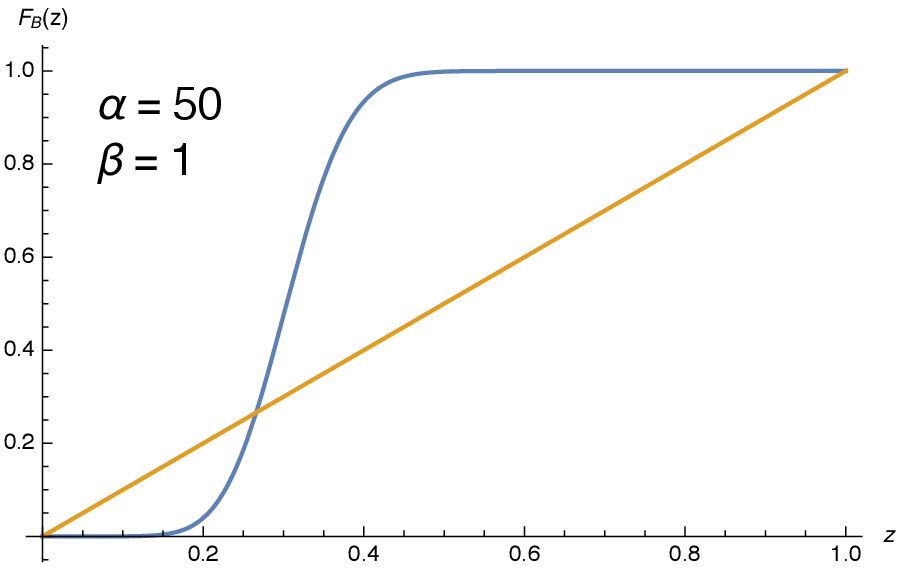}
\caption{$\alpha=50$, $\beta = 1$}
\end{subfigure}

\caption{The plots of $F_B(z)$ for given values of $\alpha$ and $\beta$ are in blue.  The function $z \mapsto z$ is plotted in orange for comparison.  Notice that as $\alpha$ increases, the approximation of $F_B(z)$ by $z$ gets worse.}
\label{fig: FBexamples}
\end{figure}

\begin{figure}[h]
\begin{subfigure}{.4\textwidth}
\includegraphics[scale=.7]{alpha10beta1.eps}
\caption{$\alpha=10$, $\beta = 1$}
\end{subfigure}
\hspace{15mm}
\begin{subfigure}{.4\textwidth}
\includegraphics[scale=.7]{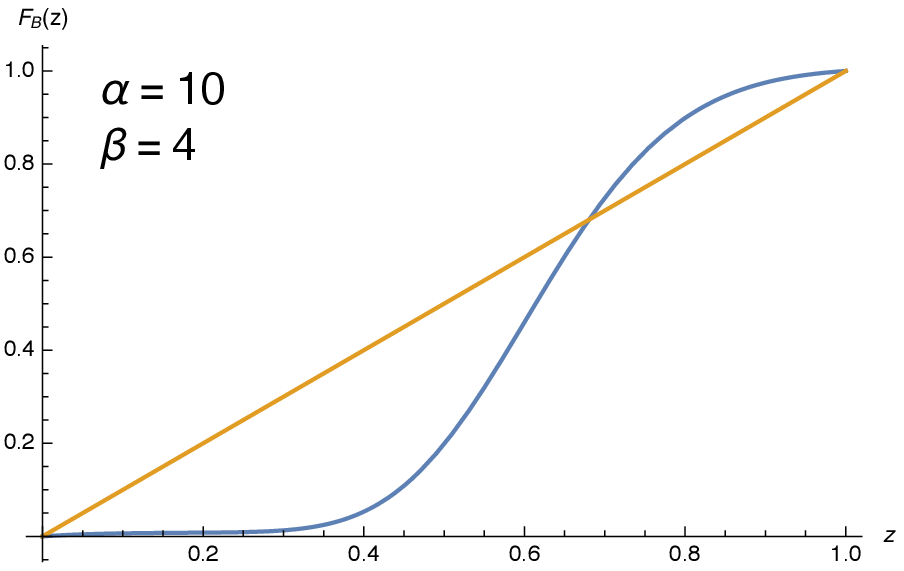}
\caption{$\alpha=10$, $\beta = 4$}
\end{subfigure}
\\
\begin{subfigure}{.4\textwidth}
\bc
\includegraphics[scale=.7]{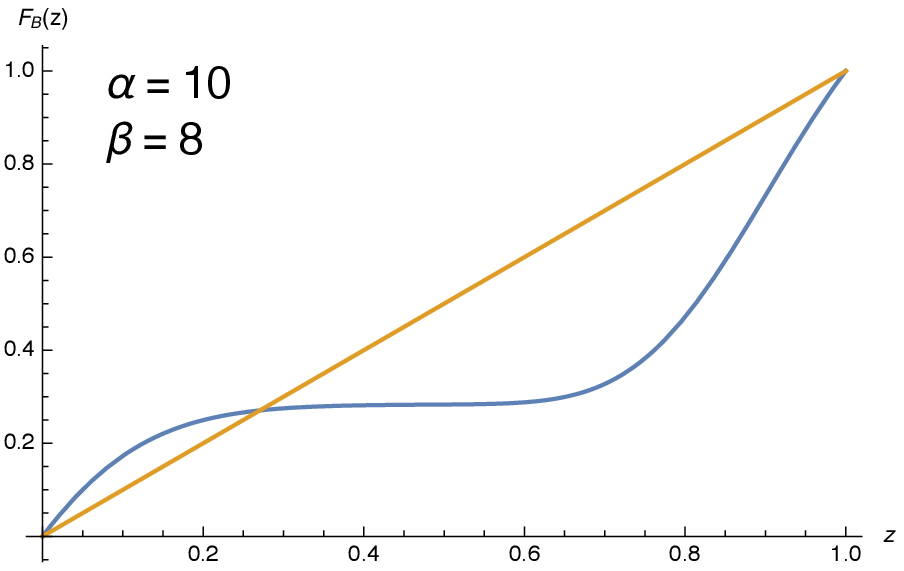}
\caption{$\alpha=10$, $\beta = 8$}
\ec
\end{subfigure}

\caption{For a fixed $\alpha$, note that as $\beta$ increases, the shape of $F_B(z)$ changes, but the maximum deviation from $z$ remains approximately the same.}
\label{fig: FBexamples2}
\end{figure}

\appendix

\section{Bounding the truncation error in the special case $\alpha = 1$} \label{app: bound1}
As mentioned above, when $\alpha = 1$ it is possible for us to achieve better bounds on the truncation error using methods similar to those in \cite{Weibull}.
\begin{theorem}
Let $F'_B(z)$ be as in Theorem \ref{thm: series} with $\alpha = 1$.
\begin{enumerate}
\item For $M \geq \frac{\log{2}\log{B}}{4\pi^2}$, the contribution to $F'_B(z)$ from the tail of the expansion (from the terms with $k \geq M$ in (\ref{eq: SimpDeriv}))  is at most
\begin{equation}
\frac{4(\pi ^2 + \log{B})}{\pi \sqrt{\log{B}}} M \exp \left( \frac{-\pi^2 M}{\log{B}} \right).
\end{equation}
\item For an error of at most $\epsilon$ from ignoring the terms with $k \geq M$ in (\ref{eq: SimpDeriv}), it suffices to take
\begin{equation}
M \ = \ \frac{h+\log{h} + 1/2}{a}
\end{equation}
where $a = \frac{\pi^2}{\log{B}}$, $h = \max \left(6, -\log{\frac{a\epsilon}{C}} \right)$, and $C = \frac{4(\pi^2 + \log{B})}{\pi\log{B}}$.
\end{enumerate}
\end{theorem}

\begin{proof}
\text{ } \\
\begin{enumerate}
\item
As stated, we estimate the contribution to $F'_B(z)$ from the tail when $\alpha = 1$.  Let
\begin{equation} \label{eq: Mtail}
E_M(z) \ := \ \frac{2}{\Gamma(1)} \sum_{k = M}^{\infty} \Re \left( e^{2\pi ik(\log_{B}{\beta-z})} \Gamma \left( 1 + \frac{-2\pi i k}{\log{B}} \right) \right)
\end{equation}
where $\Gamma(1+iu) = \int_{0}^{\infty} e^{-x} x^{iu} dx$ with $u = \frac{-2\pi i k}{\log{B}}$ in our case.  We note that as $u$ increases, there is more oscillation, which means the integral would achieve a smaller value when $u$ increases.  Since $|e^{i\theta}| = 1$, when we take the absolute values inside the sum we get $|e^{2\pi ik(\log_{B}{\beta-z})}| = 1$.  Thus it is safe to ignore this term in computing the upper bound.

Using the fact that $|\Gamma(1+ix)|^2 = \frac{\pi x}{\sinh (\pi x)}$, we have from (\ref{eq: Mtail}):

\begin{align}\label{eq:ReducedMTail}
\nonumber |E_M(z)| &\ \leq \ \frac{2}{\Gamma(1)} \sum_{k = M}^{\infty} \left| e^{2\pi ik(\log_{B}{\beta-z})} \right| \left| \Gamma \left( 1 + \frac{-2\pi i k}{\log{B}} \right) \right| \\
\nonumber &\ \leq \ \frac{2\sqrt{2}\pi}{\sqrt{\log{B}}} \sum_{k=M}^{\infty} \sqrt{\frac{k}{\sinh \left( \frac{2\pi^2k}{\log{B}} \right) }} \\
\nonumber &\ = \ \frac{2\sqrt{2}\pi}{\sqrt{\log{B}}} \sum_{k=M}^{\infty} \sqrt{\frac{2k^2}{\exp \left( \frac{2\pi^2k}{\log{B}} \right) - \exp \left( \frac{-2\pi^2k}{\log{B}} \right) }} \\
&\ \leq \ \frac{4\pi}{\sqrt{\log{B}}} \sum_{k=M}^{\infty} \sqrt{k^2 / \exp\left(\frac{2\pi^2k}{\log B}\right)}.
\end{align}
Here we have overestimated the error by disregarding the difference in the denominator, which is very small when $k$ is big. Let $u = \exp \left( \frac{2\pi^2k}{\log{B}} \right)$.  For $\frac{1}{u-{1/u}} < \frac{2}{u}$, we must get $u \geq \sqrt{2}$, which means $\exp \left( \frac{2\pi^2k}{\log{B}} \right) \geq \sqrt{2}$.  Solving this gives us $k \geq \frac{\log{2}\log{B}}{4\pi^2}$, which will help us simplify the denominator as we can assume $M$ exceeds this value and $k \geq M$.  We can now substitute this bound into (\ref{eq:ReducedMTail}) to simplify further:
\begin{align}
\nonumber |E_M(z)| &\ \leq \ \frac{4\pi}{\sqrt{\log{B}}}  \sum_{k=M}^{\infty} \frac{\sqrt{2}k}{\exp \left( \frac{\pi^2k}{\log{B}} \right)} \\
&\ \leq \ \frac{4\pi}{\sqrt{\log{B}}} \int_{M}^{\infty} m \exp \left( \frac{-\pi^2m}{\log{B}} \right) dm.
\end{align}
We let $a = \frac{\pi^2}{\log{B}}$ and apply integration by parts to get
\begin{align}
\nonumber|E_M(z)| &\ \leq \ \frac{4\pi}{\sqrt{\log{B}}} \frac{1}{a^2}\left( aMe^{-aM} + e^{-aM} \right) \\
\nonumber &\ \leq \ \frac{4\pi}{\sqrt{\log{B}}} \frac{a+1}{a}Me^{-aM} \\
&\ = \ \frac{4\pi(a+1)}{a\sqrt{\log{B}}} Me^{-aM},
\end{align}
which simplifies to
\begin{equation}
|E_M(z)| \ \leq \ \frac{4(\pi^2 + \log{B})}{\pi \sqrt{\log{B}}} M \exp \left( \frac{-\pi^2M}{\log{B}} \right),
\end{equation}
proving part (1). \\
\\
\item
Let $C = \frac{4(\pi^2 + \log{B})}{\pi\log{B}}$ and $a = \frac{\pi^2}{\log{B}}$ as before.  We want
\begin{equation}
CMe^{-aM} \ \leq \ \epsilon.
\end{equation}
We will do this by iteratively expanding to improve the bounds.
Let $v = aM$, then
\begin{equation}
\frac{C}{a}ve^{-v} \ \leq \ \epsilon \Longleftrightarrow ve^{-v} \ \leq \ \frac{a\epsilon}{C}.
\end{equation}
We carry out a change of variables one more time, letting $h = -\log{\frac{a\epsilon}{C}}$ and expanding $v$ as $v = h + x$. This leads to
\begin{align} \label{eq: expansion}
\nonumber &ve^{-v} \ \leq e^{-h}\ \\
&\longleftrightarrow\frac{h+x}{e^x} \ \leq \ 1.
\end{align}
Now we note that by expanding $v$ in this way, solving for $x$ is equivalent to solving for $v$ , which is equivalent to solving for $M$. We guess $x = \log{h} + \frac{1}{2}$ then the left-hand-side of \ref{eq: expansion} becomes:
\begin{equation}
\frac{h + \log{h} + 1/2}{he^{1/2}} \ \leq \ 1 \leftrightarrow h + \log{h} + 1/2 \ \leq \ he^{1/2}.
\end{equation}
Now what we want to do is to determine the value of $h$ so that $\log{h} \leq h/2$ since this ensures the inequality above would hold. The aforementioned inequality gives $h \leq e^{h/2}$ or $h^2 \leq e^h$.  Since for $h$ positive, $e^h \geq \frac{h^3}{3!}$, it is sufficient to choose $h$ such that $h^2 \leq h^3/6$ or $h \geq 6$.  For $h \geq 6$,
\begin{equation}
h + \log{h} + \frac{1}{2} \ \leq \ h + \frac{h}{12} + \frac{h}{2} \ = \ \frac{19h}{12} \ \approx \ 1.5883h.
\end{equation}
As $he^{1/2} \approx 1.64872h$, a sufficient cutoff for $M$ in terms of $h$ for an error of at most $\epsilon$ is
\begin{equation}
M \ = \ \frac{h+\log{h} + 1/2}{a}
\end{equation}
with $a = \frac{\pi^2}{\log{B}}$, $h = \max \left( 6, -\log{\frac{a\epsilon}{C}} \right)$.

\end{enumerate}
\end{proof}

\ \\

\end{document}